\documentclass[12pt,a4paper]{amsart}
\usepackage{graphicx}
\usepackage[top=1in,bottom=1in,left=1in,right=1.2in,marginparwidth=1.0in]{geometry}

\usepackage{amssymb,amstext,amsmath,amsthm}
\usepackage[utf8]{inputenc}
\usepackage[linesnumbered, ruled, norelsize]{algorithm2e}
\usepackage{epstopdf}
\usepackage{color}
\usepackage[normalem]{ulem}
\newcommand{\M}{\mathcal{M}}

\newcommand{\set}[1]{\left\{#1\right\}}

\newcommand{\yd}{y^\delta}

\newcommand{\xdag}{x^\dagger}

\newcommand{\xad}{x_{\alpha}^\delta}

\newenvironment{frcseries}{\fontfamily{frc}\selectfont}{}

\newcommand{\1}{\ell^1}
\newcommand{\2}{\ell^2}
\newcommand{\3}{\ell^\infty}
\newcommand{\sgn}{\mathop{\mathrm{sgn}}}

\newcommand{\xbe}{x_{\alpha,\eta}^\delta}

\def\IN{\mathbb{N}}

\newtheorem{thm}{Theorem}

\newtheorem{lemma}{Lemma}

\newtheorem{assumption}{Assumption}

\theoremstyle{definition}

\theoremstyle{remark}
\newtheorem{remark}{Remark}
\newtheorem{xmpl}{Example}

\title[Regularization with $\ell^0$-term and a convex penalty]{Tikhonov regularization with $\ell^0$-term complementing a convex penalty: $\ell^1$-convergence under sparsity constraints}
\author{Wei Wang}
\address{College of Mathematics, Physics and Information Engineering, Jiaxing University, Zhejiang, China 314001}
\email{weiwang@mail.zjxu.edu.cn}
\author{Shuai Lu}
\address{School of Mathematical Sciences, Fudan University, Shanghai, China 200433}
\email{slu@fudan.edu.cn}
\author{Bernd Hofmann}
\address{Faculty of Mathematics, Chemnitz University of Technology,
  09107 Chemnitz,  Germany}
\email{bernd.hofmann@mathematik.tu-chemnitz.de}
\author{Jin Cheng}
\address{School of Mathematical Sciences, Fudan University, Shanghai, China 200433}
\email{jcheng@fudan.edu.cn}

\date{\today}

\begin{document}
\begin{abstract}
Measuring the error by an $\ell^1$-norm, we analyze under sparsity assumptions an $\ell^0$-regularization approach, where the penalty in the Tikhonov functional is complemented by a general stabilizing convex functional. In this context, ill-posed operator equations $Ax=y$ with an injective and bounded linear operator $A$ mapping between $\ell^2$ and a Banach space $Y$ are regularized. For sparse solutions, error estimates as well as linear and sublinear convergence rates are derived based on a variational inequality approach, where the regularization parameter can be chosen either a priori in an appropriate way or a posteriori by the sequential discrepancy principle. To further illustrate the balance between the $\ell^0$-term and the complementing convex penalty, the important special case of the $\ell^2$-norm square penalty is investigated showing explicit dependence between both terms. Finally, some numerical experiments verify and illustrate the sparsity promoting properties of corresponding regularized solutions.

%based on an adopted prime-dual algorithm illustrate the sparsity promoting properties of corresponding regularized solutions.
\end{abstract}

\maketitle

\section{Introduction}
\label{se1}

Variational sparsity regularization has gained significant attention in the past years, because of its wide field of applications in mathematical methods for imaging, natural sciences and finance. We refer for details to the corresponding chapters in \cite{Scherz11,Scherzetal09,SKHK12} and diverse research papers like \cite{BoHo10,BoHo13,BreLor09,BDDML09,Kindermann16,Lorenz08,Ramlau08,RamRes10,RamTe10}. There are many papers on $\ell^1$-regularization of ill-posed
operator equations under sparsity constraints using the $\ell^1$-norm as penalty term in a Tikhonov regularization approach (see \cite{BFH13} and references therein), but the variety of sparsity promoting penalties is much broader. In the elastic-net technique (cf., e.g.,~\cite{ChenHofZou17,JLS09,ZH05}) a multi-parameter version is implemented which uses a linear combination of $\ell^1$-norm and $\ell^2$-norm square. On the other hand, $\ell^p$ seminorms with $p<1$ strongly support the
sparsity selection of approximate solutions in variational regularization, but the penalties are no longer convex for such $p$. Moreover, the extremal case of using an $\ell^0$-term alone fails, because the required stabilizing effect of the penalty for overcoming the ill-posedness is completely lost in that case. Therefore, in \cite{WLMC13} a linear combination of the $\ell^0$-term with the stabilizing $\ell^2$-norm square was suggested,
which extends the idea of elastic-net regularization. Affirmative analytical and numerical results presented in \cite{WLMC13} showed the utility of such an approach. For obtaining convergence rates a replacement of
the common use of adapted source conditions in Banach space regularization (cf.~\cite{BurOsh04,Resm05}) by the variational inequality approach \cite{HKPS07,Flemmingbuch12,Fle13,HofYam10} was successful in this context. Unfortunately, such kind
of regularization including an $\ell^0$-term is often connected with an enormous computational effort. Nevertheless, it is worthwhile to consider the error analysis of this approach in all facets, and
we complement the analysis in this paper by deriving $\ell^1$-norm error results under the extended situation that general convex functionals are added to the  $\ell^0$-term in the penalty.

Let $\widetilde A \in \mathcal{L}(\widetilde X,Y)$ be an {\sl injective} and bounded linear operator mapping between the infinite dimensional separable Hilbert space $\widetilde X$ with inner product $\langle \cdot,\cdot\rangle_{\widetilde X}$ and norm $\|\cdot\|_{\widetilde X}$ and the infinite dimensional Banach space $Y$ with norm $\|\cdot\|$
possessing the dual $Y^*$ with norm $\|\cdot\|_{Y^*}$. The symbol $\langle \cdot, \cdot \rangle_{Y^* \times Y}$ denotes the corresponding dual pairing
between $Y$ and its dual.
We assume that the range of $\widetilde A$ is not closed in $Y$, i.e.~${\rm Range}(\widetilde A) \not= \overline {{\rm Range}(\widetilde A)}^{\,Y}$,  which is equivalent to the fact that the Moore-Penrose inverse
${\widetilde A}^\dagger: {\rm Range}(\widetilde A)\oplus {\rm Range}(\widetilde A)^\perp \subset Y \to \widetilde X$ is an unbounded linear operator.
Hence the operator equation
\begin{equation} \label{eq:tildeeq}
\widetilde A\,x\,=\,y, \qquad x \in \widetilde X,\quad y \in Y,
\end{equation}
expressing the model of a linear inverse problem with the forward operator $\widetilde A$, is ill-posed. Consequently, for noisy data $\yd \in Y$ replacing $y \in {\rm Range}(\widetilde A)$  in (\ref{eq:tildeeq}) and satisfying the inequality
\begin{equation}\label{eq:noise}
\|y-y^\delta\| \leq \delta,
\end{equation}
with given noise level $\delta>0$, solutions need not exist. When they exist
the solution may be strongly perturbed even if $\delta$ is small.

For a prescribed orthonormal basis $\{u_i\}_{i=1}^\infty$ in $\widetilde X$ with elements
$\widetilde x= \sum \limits_{i=1}^\infty x_i u_i$ we consider the infinite sequences $x=\{x_i\}_{i=1}^\infty$  collecting the Fourier coefficients $x_i:=\langle \widetilde x,u_i \rangle_{\widetilde X}$.
Our focus is on \emph{sparse solutions} $\widetilde \xdag$
to equation (\ref{eq:tildeeq}) in the sense that  $\xdag=\{x^\dagger_i\}_{i=1}^\infty:=\{\langle \widetilde \xdag,u_i \rangle_{\widetilde X}\}_{i=1}^\infty \in \ell^0$, where as usual the symbol $\ell^0$
denotes the set of infinite sequences with a finite number of nonzero components and
we set in this context $\|x\|_{\ell^0}:=\sum \limits_{i=1}^\infty \sgn(|x_i|)$, for
$$\sgn(z):={\scriptscriptstyle  \begin{cases}\;\; 1 \; \mbox{if} \;\; z>0,\\\;\; 0 \; \mbox{if} \;\; z=0, \\ -1 \; \mbox{if} \;\; z<0. \end{cases} }$$

Because the sparsity can be lost when the basis changes,
the choice of the basis $\{u_i\}_{i=1}^\infty$ must be done appropriately. Then we can consider instead of (\ref{eq:tildeeq}) the ill-posed equation
\begin{equation} \label{eq:opeq}
 A\,x\,=\,y, \qquad x \in X:=\2,\quad y \in Y,
\end{equation}
with uniquely determined solution $\xdag$, by exploiting the unitary synthesis operator
$L: \2 \to \widetilde X$ defined as $\widetilde x= Lx$ and the linear bounded and injective composition operator
$A:=\widetilde A \circ L$ with non-closed range, i.e.~${\rm Range}(A) \not= \overline {{\rm Range}(A)}^{\,Y}$.
For sequences $x=\{x_i\}_{i=1}^\infty$ we will denote by $\|x\|_{\ell^p}:=\left( \sum \limits_{i=1}^\infty |x_i|^p\right)^{1/p}$ the norms in the Banach spaces $\ell^p,\;1 \leq p<\infty$, and by $\|x\|_{\ell^\infty}:=\sup \limits_{k \in \mathbb{N}} |x_i|$  the norm in  $\ell^\infty$. For simplicity we use the symbol
$\langle \cdot,\cdot \rangle$ without subscript for the inner product in $\2$.

To find stable approximate solutions to equation (\ref{eq:opeq}) based on noisy data $\yd \in Y$ we use
the Tikhonov type regularization
\begin{equation} \label{eq:Tiktwo}
\Phi_{\alpha}^\delta(x):= \frac{1}{q}\|Ax-y^\delta\|^q + \alpha\,\Omega(x) \to \min, \quad \mbox{subject to}\quad x \in X=\ell^2,
\end{equation}
with a regularization parameter $\alpha>0$ and minimizers denoted by $\xad$, for a penalty functional $\Omega$ of the form
\begin{equation} \label{eq:Omegatwo}
\Omega(x):=  \|x\|_{\ell^0}+ \mathcal{R}(x),
\end{equation}
where a proper, non-negative, {\sl convex} and {\sl stabilizing} functional $\mathcal{R}$ complements the $\ell^0$-term in the penalty.
For the exponent $q$ in the misfit term we restrict to the interval $1<q<\infty$ and exclude the case $q=1$ in order to avoid exact penalization (cf.~\cite{AnzHofMat13}).
We note that $\Omega$ from (\ref{eq:Omegatwo}) is a non-convex functional due to the non-convexity of the contained term $\|\cdot\|_{\ell^0}$. For extended discussion on non-convex minimization we refer to \cite{Grasm10,Gras10}.

As an important special case we have in mind $\mathcal{R}(x):=\frac{\eta}{2}\|x\|^2_{\2}$, for a prescribed constant $\eta>0$, such that
\begin{equation} \label{eq:Omegaeta}
\Omega(x):=  \|x\|_{\ell^0}+ \frac{\eta}{2}\|x\|^2_{\2}.
\end{equation}
Regularized solutions for that penalty have been comprehensively discussed in \cite{WLMC13}. It has been shown ibid that, for every $\eta>0$, the functional $\Omega$ from (\ref{eq:Omegaeta}) is {\sl stabilizing} and
{\sl lower semi-continuous} on sublevel sets of $\ell^0$ which belong to a fixed ball in $X$. Consequently, the minimizers $\xad \in \ell^0 \subset X$ of (\ref{eq:Tiktwo}) exist for all $\alpha>0$ and are stable
with respect to further perturbations of the data $\yd$. Measuring the error by an $\ell^2$-norm, assertions on variational source conditions and convergence including rates
of the regularized solutions $\xad$ to the exact solution $\xdag \in \ell^0$  have also been given in \cite{WLMC13}.

The goal of this paper is to complement the results presented in \cite{WLMC13} from $\ell^2$-norm errors to $\ell^1$-norm errors in the case (\ref{eq:Omegaeta}) and to
extend assertions to penalties (\ref{eq:Omegatwo}) with general convex functionals $\mathcal{R}$. For sparse solutions $\xdag \in \ell^0$  and wide classes of functionals $\mathcal{R}$ we will verify the linear convergence rate
\begin{equation} \label{eq:linearrate}
\|\xad-\xdag\|_{\1}=\mathcal{O}(\delta) \quad \mbox{as} \quad \delta \to 0,
\end{equation}
provided that the regularization parameters $\alpha>0$ are chosen appropriately. Precisely, error estimates and linear convergence rates are derived based on a variational source condition approach, where the regularization parameter can be chosen either a priori in an appropriate way or a posteriori by the sequential discrepancy principle. If the convex functional $\mathcal{R}$ is chosen inappropriate, then at least sublinear convergence
rates can be proven based on the method of approximate source conditions \cite{HeinHof09,Hof06}.

The paper is organized as follows: In Section \ref{se2} we collect basic assumptions and provide a variational inequality with respect to the $\ell^1$-norm which does not appear in the penalty functional $\Omega(x)$. Error estimates for an {\it a priori} parameter choice rule and the sequential discrepancy principles are carried out in Section~\ref{se3} for general stabilizing functionals $\mathcal{R}(x)$.
To further illustrate the balance between the $\ell^0$-term and the complementing convex penalty, the important special case of the $\ell^2$-norm square penalty is investigated in Section \ref{se4} showing explicit dependence between both penalty terms.
Finally some numerical examples and extended discussion are presented in Section~\ref{se5} to verify the sparsity promoting properties of our proposed regularization schemes and other related aspects.

\section{A variational source condition approach}\label{se2}
We first present the main assumptions below.
\begin{assumption} \label{ass:basic}
\begin{itemize} \item[]
\item[(a)] Let $A: X=:\ell^2 \to Y$ in equation (\ref{eq:opeq}) be an injective and bounded linear operator mapping from the Hilbert space $X$ to the Banach
space $Y$ with an adjoint operator  $A^*: Y^* \to X$. Moreover, let ${\rm Range}(A) \not= \overline {{\rm Range}(A)}^{\,Y}$ such that the equation (\ref{eq:opeq}) is ill-posed.
\item[(b)] Let $y \in {\rm Range}(A)$ be given such that the uniquely determined solution $\xdag$ to equation (\ref{eq:opeq}) is sparse, i.e.~$\xdag \in \ell^0$. Moreover, let \[I := \{i\in \mathbb{N}:\, x_i^\dag \neq 0\} \not=\emptyset \] be the finite index set of non-zero components of $x^\dag$. In this context, we set $|I|:={\rm card}(I)=\|\xdag\|_{\ell^0}>0$ and denote by $\Sigma(\xdag):={\rm span}(\{e^{(i)}\}_{i \in I})$ the corresponding $|I|$-dimensional subspace of non-zero components,  where $e^{(i)}:=(0,0,...,0,1,0,...)$, with $1$ at the $i$-th position for $i \in \IN$, designate  the elements of standard orthonormal basis in $\ell^2$.
\item[(c)] Let $\mathcal{R}: X=\ell^2 \to [0,\infty]$ be a convex and lower semi-continuous functional, where  we denote by $\mathcal{D}(\mathcal{R})=\{x \in X: \mathcal{R}(x)<\infty\}$ the proper domain of $\mathcal{R}$ and by  $\partial \mathcal{R}(x) \subseteq X$ the subdifferential of the functional $\mathcal{R}$
at the point $x \in \mathcal{D}(\mathcal{R})$.
\item[(d)] The functional $\mathcal{R}$ is assumed to be  {\rm stabilizing} in the sense
that the sublevel sets $$\mathcal{S}^\mathcal{R}(c):=\{x \in X:\, \mathcal{R}(x) \le c\} $$ are weakly sequentially pre-compact in $X$ for all $c \ge 0$. This means that every
infinite sequence $\{x_n\}_{n=1}^\infty \subset \mathcal{S}^\mathcal{R}(c)$ possesses a subsequence which has a weak limit in $X$.
\end{itemize}
\end{assumption}

\begin{remark} \label{rem:ass2}{\rm
In the Hilbert space $X=\ell^2$ item (d) means that the functional $\mathcal{R}$ is {\sl weakly coercive}, i.e.~$\|x\|_{\2} \to \infty$ implies $\mathcal{R}(x) \to \infty$. For given $c \ge 0$ there
is a radius $r(c)>0$ such that $\sup \limits_{x \in \mathcal{S}^\mathcal{R}(c)}\|x\|_{\2} \le r(c)$.
} \hfill\fbox{}
\end{remark}

\begin{assumption} \label{ass:basic2}
 For all $i\in \IN$ we assume that there exist $\omega_i \in Y^*$ such that $e^{(i)}= A^*\omega_i$.
 Consequently, there holds
 $x_i=\langle e^{(i)},x\rangle=\langle \omega_i,Ax\rangle_{Y^* \times Y}\,$ for all $x=(x_1,x_2,...) \in \2$.
\end{assumption}

\begin{remark} {\rm
Assumption~\ref{ass:basic2} was already suggested in \cite{BreLor09,Grasm09,Lorenz08}. In a recent paper \cite{AHR13} it was shown that the set of range conditions $e^{(i)} \in {\rm Range}(A^*)$, $i \in \IN$, which is equivalent to $u_i \in {\rm Range}(\widetilde A^*)$ for all elements of the orthonormal basis $\{u_i\}_{i=1}^\infty$ in $\widetilde X$, applies for important classes of linear inverse problems. Precisely, Assumption~\ref{ass:basic2} can
be reinterpreted as a requirement on the `smoothness' of the basis elements $u_i$ with respect to the forward operator $\widetilde A$, and we refer to \cite{FlemHeg15} for consequence in the case of `non-smooth' basis elements.

Under Assumption~\ref{ass:basic2} we have for all $i \in I$
\begin{align*}
|x_i-x_i^\dag|=|\langle e^{(i)},x-x^\dag\rangle|
=|\langle \omega_i,A(x-x^\dag)\rangle_{Y^* \times Y}|\leq
\|\omega_i\|_{Y^*} \|A(x-x^\dag)\|,
\end{align*}
and hence
\begin{equation} \label{eq:e1}
\sum \limits _{i \in I} |x_i-\xdag_i| \le  K\, \|A(x-\xdag)\| \qquad \mbox{for all} \quad x \in \2
\end{equation}
with a constant $K=\sum \limits_{i \in I} \|\omega_i\|_{Y^*}$.
} \hfill\fbox{}
\end{remark}

\begin{remark} {\rm
If Assumption~\ref{ass:basic2} fails, then the term $\sum_{i\in I}|x_i^\dag-x_i|$ can be, as an alternative to (\ref{eq:e1}),  estimated from above by taking into the account that for any injective $A: \2 \to Y$, which maps also continuously from $\1$ to $Y$, the {\sl modulus of injectivity} (cf.,~e.g.,~\cite[\S3.2]{HMS08} and \cite[\S2.4]{MatHof08}) $j(A,\Sigma(\xdag))$ defined as
$$ j(A,\Sigma(\xdag)):=\inf \limits _{0\not=x \in \Sigma(\xdag)} \frac{\|Ax\|}{\|x\|_{\1}}>0$$
is a finite positive number. This is similar to ideas of the proofs of Theorems~14 and 15 in \cite{GrasHaltSch08}.  Precisely, we can estimate
$$  \sum \limits_{i \in I} |x_i^\dagger - x_i| \leq \frac{1}{j(A,\Sigma(\xdag))}\,\|A (P_{I} x-\xdag)\| \qquad \mbox{for all}\quad x \in \ell^1, $$
where $P_I$ denotes the projection to the subspace $\Sigma(\xdag)$. Thus we obtain, for all $x \in \1$,
\begin{align*}
 \sum \limits_{i \in I} |x_i^\dagger - x_i| & \leq \frac{1}{j(A,\Sigma(\xdag))}\,\left(\|A (x-\xdag)\| +\|A (I-P_{I})x\|\right) \\
 &
 \le \frac{1}{j(A,\Sigma(\xdag))}\,\left(\|A (x-\xdag)\| +\|A\|_{\mathcal L(\1,Y)}\|(I-P_{I})x\|_{\1}\right) \\
& \le c_1 \,\|A (x-\xdag)\| +c_2\, \sum \limits_{i \notin I}|x_i|
\end{align*}
with positive constants $c_1$ and $c_2$. As we will see, however, handling of the term  $\sum \limits_{i \notin I}|x_i|$  becomes a serious difficulty in the further estimation process.
Hence, rates results are still missing when Assumption~\ref{ass:basic2} is violated.

} \hfill\fbox{}
\end{remark}

\begin{assumption} \label{ass:basic3}
Let $\xdag \in \mathcal{D}(\mathcal{R})$ and $\partial \mathcal{R}(\xdag) \not=\emptyset$. Assume that there exists some element $\xi^\dagger \in \partial \mathcal{R}(\xdag)$
satisfying a source condition
\begin{equation} \label{eq:e2}
\xi^\dagger=A^*w,\quad \mbox{for some}\quad w \in Y^*.
\end{equation}
\end{assumption}

\begin{remark} \label{rem:ass3}{\rm
If we have, for some $\xi^\dagger \in \partial \mathcal{R}(\xdag)$, sparsity in the sense that $\xi^\dagger \in \ell^0$, then Assumption~\ref{ass:basic2} implies that (\ref{eq:e2}) and hence Assumption~\ref{ass:basic3} hold, because with $\xdag \in \ell^0$
$$\xi^\dagger= \sum \limits_{k \in \IN:\,\xi^\dagger_k \not=0} \xi^\dagger_k\,e^{(k)}=A^* \left( \sum \limits_{k \in \IN:\,\xi^\dagger_k \not=0} \xi^\dagger_k\,\omega_k \right). $$
This is in particular under Assumption~\ref{ass:basic} the case whenever the implication
\begin{equation} \label{eq:implic}
\xdag \in \ell^0 \quad \Longrightarrow \quad \xi^\dagger \in \ell^0
\end{equation}
holds true.
} \hfill\fbox{}
\end{remark}

We provide two examples when Assumption \ref{ass:basic3} is satisfied or violated.
\begin{xmpl} \label{ex:implic}
Consider the family of functionals
$$\mathcal{R}(x):= \frac{\eta}{p}\|x\|^p_{\ell^p}, \qquad \eta>0,$$
with exponents $1 < p \le 2$, where the subdifferential $\partial \mathcal{R}(\xdag)$ is always a singleton such that
$$\xi^\dagger_k= \eta\,|\xdag_k|^{p-1}\sgn(\xdag_k),\;k \in \IN. $$
For such exponents $p$ the functionals $\mathcal{R}$ are convex, lower semi-continuous and stabilizing in the sense of item (d) of Assumption~1.
Consequently, Assumption~3 is applicable, because $\xdag \in \ell^0$ implies $\xi^\dagger \in \ell^0$  and hence (\ref{eq:e2}) is satisfied. Note that therefore for the penalty functional
$$\Omega(x):=\|x\|_{\ell^0}+\frac{\eta}{p}\|x\|^p_{\ell^p}, \qquad 1<p \le2,\;\;\eta>0,$$
the main Theorem~\ref{thm:main} below applies and provides us with linear convergence rates of the Tikhonov regularization when all other assumptions are fulfilled.
\end{xmpl}

\begin{xmpl} \label{ex:noimplic}
Consider alternatively the family of functionals
$$\mathcal{R}(x):=\frac{1}{2}\|Bx\|_{\2}^2$$
 with a bounded linear operator $B: \2 \to \2$ such that
$Be^{(1)}:=\sum \limits_{i=1}^\infty\lambda_i\,e^{(i)}\;$ with $\;\sum \limits_{i=1}^\infty\lambda_i^2<\infty,\;\lambda_i>0\;(i=1,2,...)$, and  $\;Be^{(k)}:=e^{(k)}$ for all $k \ge 2$.
Then we have $\xi^\dagger=B^*B \xdag$, where always $\xi^\dagger \notin \ell^0$,
although the solution $0 \not=\xdag \in \ell^0$ is sparse.
\end{xmpl}

The basic Lemma~\ref{lem_2:1} below follows directly from Assumptions~\ref{ass:basic} -- \ref{ass:basic3} and yields a variational inequality which, acting as a variational source condition, allows for convergence rates if an $\ell^0$-term
complemented by a convex term forms the penalty functional of Tikhonov regularization.

\begin{lemma}\label{lem_2:1}
Let
\begin{equation}\label{eq:M}
 \M:=\{ x \in \ell^0:\,\|x\|_{\2} \leq \rho\}
\end{equation}
for a prescribed sufficiently large real value $\rho \ge 1$.
Then under Assumptions~\ref{ass:basic}, \ref{ass:basic2} and \ref{ass:basic3} there exist constants $0<\beta \le 1$ and $c_0  \ge 0$ depending on $A$, $\mathcal{R}$ and $\xdag \in \ell^0$ such that the variational inequality
\begin{equation}\label{eq_2:1}
\beta\,\|x-x^\dag\|_{\ell^1} \leq  \|x\|_{\ell^0} - \|\xdag\|_{\ell^0}+ \mathcal{R}(x)-\mathcal{R}(\xdag)+ c_0\,\|A(x-x^\dag)\| \;\; \mbox{for all} \;\;  x \in \M
\end{equation}
is valid.
\end{lemma}
\begin{proof}
Since $x^{\dag}\in \ell^0$ and $I$ is a finite subset of $\IN$, we derive for all $x \in \ell^0$
\begin{align*}
& \|x\|_{\ell^0}-\|x^\dag\|_{\ell^0} \nonumber \\
& = \sum_{i\in I}\left[{\rm sgn}(|x_i|)-{\rm sgn}(|x_i^\dag|)\right]
+\sum_{i\notin I}{\rm sgn}(|x_i|). \nonumber
\end{align*}
Now,
\begin{equation*}%\label{eq_2:2}
\begin{aligned}
 \|x-x^\dag\|_{\ell^1}&= \sum_{i\in I}|x_i^\dag-x_i| + \sum_{i\notin I}|x_i|\\
 & \leq  \sum_{i\in I}|x_i^\dag-x_i|+ \max_{i\notin I} |x_i| \sum_{i\notin I}{\rm sgn}(|x_i|), \\
\end{aligned}
\end{equation*}
and because of  $  \max_{i\notin I} |x_i| \leq \|x\|_{\3} \leq  \|x\|_{\ell^2}\leq\rho $
we have
\begin{equation*}
\begin{aligned}
\|x-x^\dag\|_{\ell^1} &\leq
  \sum_{i\in I}|x_i^\dag-x_i| + \rho \sum_{i\notin I}{\rm sgn}(|x_i|)\\
&\leq
  \sum_{i\in I}|x_i^\dag-x_i|+ \rho \left[\|x\|_{\ell^0}-\|x^\dag\|_{\ell^0}
 -\sum_{i\in I}\left[{\rm sgn}(|x_i|)-{\rm sgn}(|x_i^\dag|)\right]
\right].
\end{aligned}
\end{equation*}
To estimate the right-hand side of the last inequality from above we  consider $J(x):= \{i\in I:\, |x_i|=0\}$
and set  $\underline \kappa:= \min_{i\in I}{|x_i^\dag|}>0$. Then we can further estimate as
\begin{equation}\label{eq_2:3}
\sum_{i\in I}\left[{\rm sgn}(|x_i^\dag|)-{\rm sgn}(|x_i|)\right]=\sum_{i\in J(x)}\left[{\rm sgn}(|x_i^\dag|)\right]
\leq \frac{1}{\underline \kappa}\sum_{i\in J(x)}|x_i^\dag-x_i|.
\end{equation}
If $J(x)$ is an empty set, the inequality (\ref{eq_2:3}) still holds true by interpreting the sum on the right-hand side as zero. We thus conclude that
\begin{align*}
& \|x-x^\dag\|_{\ell^1}  \leq \sum_{i\in I}|x_i^\dag-x_i| + \rho \left[\|x\|_{\ell^0}-\|x^\dag\|_{\ell^0} + \frac{1}{\underline \kappa}\sum_{i\in I}|x_i^\dag-x_i|
\right].
\end{align*}
As a consequence of (\ref{eq:e1}) this yields for all $x \in \M$ the inequality
\begin{equation} \label{eq:ab1}
\frac{1}{\rho}\,\|x-x^\dag\|_{\ell^1} \leq \|x\|_{\ell^0}-\|x^\dag\|_{\ell^0} + K\,\left(\frac{1}{\rho}+\frac{1}{\underline \kappa} \right) \|A(x-x^\dag)\|
\end{equation}
with the constant $K=\sum \limits_{i \in I} \|\omega_i\|_{Y^*}$.

\smallskip

For the convex functional $\mathcal{R}$ and $\xi^\dagger \in \partial \mathcal{R}(\xdag)$ we have
$$0 \le \mathcal{R}(x)-\mathcal{R}(\xdag)+\langle \xi^\dagger,\xdag-x \rangle $$
and hence
\begin{equation}\label{eq:start}
0 \le \mathcal{R}(x)-\mathcal{R}(\xdag)+|\langle \xi^\dagger,x-\xdag \rangle| \qquad \mbox{for all} \quad  x \in X=\2\,,
\end{equation}
therefore with (\ref{eq:e2})
$$0 \le \mathcal{R}(x)-\mathcal{R}(\xdag)+\|w\|_{Y^*}\,\|A(x-\xdag)\| \qquad \mbox{for all} \quad  x \in \2\,.$$
In combination with inequality (\ref{eq:ab1}) we thus obtain the variational source condition (\ref{eq_2:1}) with the constants $0<\beta=\frac{1}{\rho} \le 1$ and
$$c_0= K\,\left(\frac{1}{\rho}+\frac{1}{\underline \kappa} \right)+\|w\|_{Y^*}.$$
This completes the proof.
\end{proof}

\begin{remark} \label{rem:necessary} {\rm
From \cite[Lemma~8.21]{Scherzetal09} we know that for the existence of a constant $C>0$ satisfying  $|\langle \xi^\dagger,x-\xdag \rangle| \le C\,\|A(x-\xdag)\|$ for all $x \in X$ the
condition (\ref{eq:e2}) is even necessary.
}

\hfill\fbox{}
\end{remark}

\section{Convergence rates results}\label{se3}

As next step we are going to apply Lemma~\ref{lem_2:1} to prove convergence rates for the variant (\ref{eq:Tiktwo}) of Tikhonov regularization with the mixed penalty functional $\Omega$ from (\ref{eq:Omegatwo}) and regularized solutions $\xad$ as minimizers of $\Phi_\alpha^\delta$. This, however, requires that for all $\alpha>0$ such regularized solutions exist. This will be shown by the following lemma.

\begin{lemma}\label{lem_2:2}
For all $\alpha>0$ and $\yd \in Y$ there are minimizers $\xad$ to the Tikhonov functional $\Phi_\alpha^\delta$ from (\ref{eq:Tiktwo}).
\end{lemma}
\begin{proof}
In \cite[Lemma~2.1]{WLMC13} it was shown that the functional $\|\cdot\|_{\ell^0}$ is weakly lower semi-continuous with respect to $\ell^0$-norm bounded sequences which are weakly convergent in $\2$.
Based on that fact the proof of \cite[Theorem~2.3]{WLMC13} can be extended from the specific penalty $\Omega$ from (\ref{eq:Omegaeta}) used there to our general convex stabilizing penalty $\Omega$ from (\ref{eq:Omegatwo}). This proves the lemma.
\end{proof}

\smallskip

Now we are ready to formulate the first theorem based on the variational source condition (\ref{eq_2:1}).

\begin{thm} \label{thm:main}
Under  Assumptions~\ref{ass:basic} -- \ref{ass:basic3} the linear convergence rate
\begin{equation} \label{eq:linearrate1}
\|x_{\alpha_*}^\delta-\xdag\|_{\1}=\mathcal{O}(\delta) \quad \mbox{as} \quad \delta \to 0,
\end{equation}
holds true for the Tikhonov regularization (\ref{eq:Tiktwo}) with $1 < q <\infty$ and the penalty functional $\Omega$ from (\ref{eq:Omegatwo}) whenever
the regularization parameter $\alpha_*>0$
is chosen either a priori as $\alpha_*=\alpha_*(\delta) \sim \delta^{q-1}$ or a posteriori as $\alpha_*=\alpha_*(\delta,\yd)$ by applying the sequential discrepancy principle (see Remark~\ref{rem:SDP} below).
\end{thm}
\begin{proof}
If we can prove that $x_{\alpha_*}^\delta \in \M$, the convergence rate result  (\ref{eq:linearrate1}) of the theorem is an immediate consequence of the Theorems~1 and 2 from \cite{HofMat12}
based on the variational source condition (\ref{eq_2:1}).
Namely, an inspection of the proofs of both theorems shows that no convexity assumption is imposed on the occurring penalties, and hence the non-convex penalty $\Omega$ from (\ref{eq:Omegatwo}) is applicable. Now it remains to show  that $x_{\alpha_*}^\delta \in \M$ is true for a sufficiently large value $\rho$.

 Evidently, we have $\xad \in \ell^0$ for all $\alpha>0$, because $\|\cdot\|_{\ell^0}$ is part of the penalty functional $\Omega$.
Moreover, from $\Phi_{\alpha_*}^\delta(x_{\alpha_*}^\delta) \le \Phi_{\alpha_*}^\delta(\xdag)$ it follows that $\mathcal{R}(x_{\alpha_*}^\delta) \le \frac{\delta^q}{q\alpha_*} + \|\xdag\|_{\ell^0}+\mathcal{R}(\xdag)$. For both choices of  the
regularization parameter $\alpha_*$ under consideration we have that $\frac{\delta^q}{\alpha_*}$ tends to zero as $\delta \to 0$ (cf.~\cite[Proposition~8]{AnzHofMat13}) and is consequently bounded by a constant $\hat K>0$ whenever $0<\delta \le \overline \delta$.
For those sufficiently small $\delta>0$ we have $\mathcal{R}(x_{\alpha_*}^\delta) \le \overline K$ for  $\overline K:=\frac{\hat K}{q}+ \|\xdag\|_{\ell^0}+\mathcal{R}(\xdag)<\infty$. The stabilizing property of the functional $\mathcal{R}$ from item (d) of Assumption~\ref{ass:basic} ensures that there is a radius $r(\overline K)$ such that $\|x_{\alpha_*}^\delta\|_{\2} \le r(\overline K)$. Prescribing $\rho$ such that the inequality $\rho \ge r(\overline K)$ is valid, we finally find  $x_{\alpha_*}^\delta \in \M$ for sufficiently small $\delta>0$. This completes the proof.
\end{proof}

\smallskip

\begin{remark} \label{rem:SDP} {\rm
We say that $\alpha_*=\alpha(\delta,\yd)$ is chosen according to the {\sl sequential discrepancy principle}
if the regularization parameter is taken for sufficiently large $\alpha_0>0$ and prescribed constants $0<\nu<1,\; \tau >1,$ from a geometric sequence
\begin{equation*}
  \Delta_{\nu}:= \set{\alpha_{j}:\quad \alpha_{j}:=
    \nu^{j}\alpha_{0},\quad j=1,2,\dots},
\end{equation*}
such that $\alpha_*=\alpha_{j_*} \in \Delta_\nu$ satisfies
$$\|Ax_{\alpha_{j_*}}^\delta-\yd\| \le \tau\,\delta <  \|Ax_{\alpha_{(j_*-1)}}^\delta-\yd\| .$$
} \hfill\fbox{}\end{remark}

\medskip

Now let us consider the situation that the source condition (\ref{eq:e2}) and hence Assumption~\ref{ass:basic3} are violated. Then the estimate (\ref{eq:start}) is the starting point for applying the {\sl method of approximate source conditions}
(cf.~\cite{Hof06} and \cite{HeinHof09}) for obtaining convergence rates which are lower than (\ref{eq:linearrate}). In this context, we call $\varphi:(0,\infty) \to (0,\infty)$ an index function if it is continuous and strictly increasing with $\lim_{t \to +0} \varphi(t)=0$. A basic tool is the positive, concave and strictly decreasing {\rm distance function}
\begin{equation} \label{eq:distfct}
d_{\xi^\dagger}(R):= \min\{\|\xi^\dagger-A^*w\|_{\2}:\,w \in X,\,\|w\|_{Y^*} \le R\},\qquad R>0,
\end{equation}
satisfying $\lim_{R \to \infty} d_{\xi^\dagger}(R)=0$. When the source condition (\ref{eq:e2}) fails, then we have for all $R>0$ elements $w_R \in Y^*$ with $\|w_R\|_{Y^*}=R$ and $\zeta_R \in \2$ with $\|\zeta_R\|_{\2}=d_{\xi^\dagger}(R)$ such that $\xi^\dagger=A^*w_R+\zeta_R$.

\begin{thm}\label{pro:approx}
Under  Assumptions~\ref{ass:basic} and \ref{ass:basic2} and in the case that Assumption~\ref{ass:basic3} is violated, i.e.~the source condition (\ref{eq:e2}) fails, define the strictly concave index function
\begin{equation*}% \label{eq:varphi}
\varphi(t):=d_{\xi^\dagger}(\Psi^{-1}(t)),
\end{equation*}
including both the distance function $d_{\xi^\dagger}$ from (\ref{eq:distfct}) and the auxiliary function $\Psi(R):=\frac{d_{\xi^\dagger}(R)}{R}$, then the sublinear convergence rate
\begin{equation*} % \label{eq:phirate}
\|x_{\alpha_*}^\delta-\xdag\|_{\1}=\mathcal{O}(\varphi(\delta)) \quad \mbox{as} \quad \delta \to 0,
\end{equation*}
holds true for the Tikhonov regularization (\ref{eq:Tiktwo}) with $1 < q <\infty$ and the penalty functional $\Omega$ from (\ref{eq:Omegatwo}) whenever
the regularization parameter $\alpha_*>0$
is chosen either a priori as $\alpha_*=\alpha_*(\delta)= \frac{\delta^p}{\varphi(\delta)}$ or a posteriori as $\alpha_*=\alpha_*(\delta,\yd)$ by applying the sequential discrepancy principle.
\end{thm}
\begin{proof}
The proof is analogous to the proof of the above Theorem \ref{thm:main} and based on Theorems~1 and 2 from \cite{HofMat12}. The new aspect is to derive an inequality of the form
\begin{equation*} %\label{eq:toshow}
0 \le \mathcal{R}(x)-\mathcal{R}(\xdag)+ C\,\varphi(\|A(x-\xdag)\|), \qquad 0<C<\infty\,,
\end{equation*}
for all $x$ from a closed ball $B_{\overline r}(\xdag)$ of $X=\2$ with sufficiently large radius $\overline r$.
To show such an inequality we use (\ref{eq:start}) and estimate  $|\langle \xi^\dagger,x-\xdag \rangle|$ for the $x \in X$ under consideration from above as
\begin{align*}
|\langle \xi^\dagger,x-\xdag \rangle| &  \le |\langle w_R, A(x-\xdag)\rangle_{Y^* \times Y}+\langle \zeta_R,x-\xdag \rangle| \\
& \le \|w_R\|_{Y^*}\|A(x-\xdag)\|+d_{\xi^\dagger}(R)\overline r \le R\,\|A(x-\xdag)\|+d_{\xi^\dagger}(R)\overline r.
\end{align*}
After some calculation this implies the estimate
$$ |\langle \xi^\dagger,x-\xdag \rangle| \le (\overline r+1)\,d_{\xi^\dagger}(\Psi^{-1}(\|A(x-\xdag)\|)) $$
and completes the proof taking into the account that $$K\delta+C\varphi(\delta)=\mathcal{O}(\varphi(\delta)) \quad \mbox{as} \quad \delta \to 0\,.$$
\end{proof}

\section{The special case $\mathcal{R}(x)=\frac{\eta}{2}\|x\|^2_{\2}$} \label{se4}

In this section, we consider the special case
\[\mathcal{R}(x)=\frac{\eta}{2}\|x\|^2_{\2}, \quad \eta>0,\]
which yields the specific penalty $\Omega$ in (\ref{eq:Omegaeta}). Our focus is on the impact of the coefficient $\eta$ on the error estimates. Evidently, Assumption~\ref{ass:basic3} is always satisfied with the source element $\omega=\sum \limits_{k \in \IN:\,\xi^\dagger_k \not=0}
 \xi^\dagger_k\,\omega_k$, which implies that $\|\omega\|_{Y^*} \le K\eta\,\|\xdag\|_{\ell^\infty}$ with $K$ defined in the context of formula (\ref{eq:e1}).
In Theorem~\ref{thm:main} we have derived linear error estimates for Tikhonov regularization with the combined penalty functional $\Omega(x)$, but we have ignored factors weighting the $\ell^0$-term relative to the convex stabilizing part  in the penalty $\Omega$ as it is the factor $\eta$ in (\ref{eq:Omegaeta}).

In the sequel we exclude the singular case $\xdag=0$ and estimate the error $\|\xbe-\xdag\|_{\1}$, denoting $\xbe$ be the minimizer, under appropriate parameter choice rules. For any fixed $\eta>0$, from $\Phi_\alpha(\xbe) \leq \Phi_\alpha(\xdag)$ the following inequality hold true
$$\|\xbe\|^2_{\2} \leq \frac{2\delta^q}{q\alpha \eta}+\frac{2}{\eta}\|\xdag\|_{\ell^0}+\|\xdag\|_{\2}^2.$$
Consequently, we can find $\xbe \in \M$ for sufficiently small $\delta>0$ and
\begin{equation} \label{eq:relation}
\rho> \sqrt{\frac{2}{\eta}\|\xdag\|_{\ell^0}+\|\xdag\|_{\2}^2}.
\end{equation}
More precisely,
a priori parameter choices $\alpha_*=\alpha_*(\delta)$ and a posteriori parameter choices  $\alpha_*=\alpha_*(\delta,y^\delta)$ with some $\overline \delta>0$ yield $ \|x^\delta_{\alpha_*,\eta}\|_{\2} \leq \rho$ and hence $ x^\delta_{\alpha_*,\eta} \in \M$ for $0<\delta \leq \overline \delta$ provided that the limit condition $\delta^q/\alpha_* \to 0$ as $\delta \to 0$ holds true. Thus, the choice of a sufficiently large $\rho>0$ allows us to satisfy the premise of the following theorem for $0<\delta \leq \overline \delta$.

\begin{thm}\label{th_2:1}
For the specific penalty $\Omega$ from (\ref{eq:Omegaeta}) let items (a) and (b) of Assumption~\ref{ass:basic} and Assumption~\ref{ass:basic2} be satisfied.
Fix a sufficiently large bound $\rho>0$ for the set $\M$ from (\ref{eq:M}). Then for all triples $(\eta,\alpha,\delta)$ of positive real numbers with $\xbe \in \M$ there holds, denoting $c_0(\eta)$ by
$$c_0(\eta):=K\,\left(\frac{1}{\rho}+\frac{1}{\underline \kappa} +\eta\|x^\dagger\|_{\ell^\infty}\right),$$
the error estimate
\begin{equation} \label{eq:errgen}
 \begin{aligned}
\|\xbe-x^\dag\|_{\ell^1} \leq & \rho\left[\frac{\delta^{q}}{q\alpha}
           +  c_0(\eta)\delta+ (c_0(\eta))^{\frac{q}{q-1}}\alpha^{\frac{1}{q-1}}\right].
 \end{aligned}
 \end{equation}
For the standard a priori parameter choice $\alpha_*=\alpha(\delta):= \delta^{q-1}$, the error estimate
\begin{equation} \label{eq:Capriori}
\|x_{\alpha_*,\eta}^\delta-\xdag\|_{\1} \leq C\,\delta \quad \mbox{with} \quad C= \rho\left[\frac{1}{q}+c_0(\eta)
 +(c_0(\eta))^{\frac{q}{q-1}}\right],
\end{equation}
holds true for $\eta>0$ and $\rho>0$ satisfying the condition (\ref{eq:relation}) and for sufficiently small $\delta>0$.
\end{thm}
\begin{proof}
Since $\xbe$ is a minimizer of functional $\Phi_\alpha$,
%\begin{equation*}
%\begin{aligned}
%\Phi_\alpha(\xbe) \leq \Phi_\alpha(\xdag) \leq \frac{\delta^{q}}{q}+\alpha \mathcal{R}_\eta(\xdag).
%\end{aligned}
%\end{equation*}
inserting the inequality (\ref{eq_2:1}), we have
\begin{equation}\label{eq_Theoremineq1}
\begin{aligned}
 \frac{1}{q}\|A\xbe-y^{\delta}\|^{q}\leq&\frac{\delta^q}{q}
       + \alpha(\Omega(x^\dag)-\Omega(\xbe))\\
 \leq & \frac{\delta^{q}}{q}
       -\frac{\alpha}{\rho}\|\xbe-x^\dag\|_{\ell^1}
       +\alpha c_0(\eta)\|A(\xbe-x^\dag)\| \\
 \leq & \frac{\delta^{q}}{q}
        -\frac{\alpha}{\rho}\|\xbe-x^\dag\|_{\ell^1}
           +\alpha c_0(\eta)(\|A\xbe-y^\delta\|+\delta).
\end{aligned}
\end{equation}

%$$c_0(\eta)=\frac{1}{j_{|I|}(A)}\,\left(\frac{1}{\rho}+\frac{1}{\underline \kappa} \right)+\|w\|_{Y^*}\,\|A(x-\xdag)\|$$

As $q>1$ we use Young's inequality with the dual exponent $q'= q/(q-1)$ as $ab\leq a^q/q+b^{q'}/q',\;a,b>0,$ to obtain
 $$ \alpha c_0(\eta) \|A\xbe-y^\delta\| \leq \frac{(q-1)}{q}\, (\alpha c_0(\eta))^{\frac{q}{q-1}}\,
   + \frac{1}{q}\|A\xbe-y^{\delta}\|^{q}.$$
 Combining this with \eqref{eq_Theoremineq1} we derive (\ref{eq:errgen}) and obviously
obtain further (\ref{eq:Capriori}) for the
 a priori parameter choice $\alpha_*:=\delta^{q-1}$.
 This proves the theorem.
 \end{proof}

To keep the constant $C$ in (\ref{eq:Capriori}) small,
it seems that the bound $\rho>0$ would have to be prescribed as small as possible, which implies that $\eta$ should be large enough.
On the other hand, the value $c_0(\eta)>0$ would have to be prescribed as small as possible as well, which implies that $\eta$ should be small. As will be shown below in Section \ref{se5}, the numerical error estimates first decreases and increases later when $\eta$ decreases from large to small values.

%Therefore, we shall choose $\eta$ appropriately to control the error estimate.
%We consider the special case $q=2$, and keep the term in $C$ with respect to $\eta$
%\[
%C = \rho\left[\frac{1}{2}+c_0(\eta)+c_0(\eta)^2\right]
%\]
%\[\pi(\eta) = \left(\frac{1}{2}+\frac{1}{j_{|I|}(A)\underline{\kappa}}\right)\sqrt{ \frac{2\|x^\dag\|_{\ell^0}}{\eta}}
%+\frac{1}{j_{|I|}(A)}\sqrt{2\eta\|x^\dag\|_{\ell^0}}\|x^\dag\|_{\ell^\infty}\]
%
%this approach is limited due to the relation (\ref{eq:relation}). Indeed, $\rho$ and hence $C$ are inevitably large if either $\eta \|\xdag\|_{\ell^0}$ or $\|\xdag\|_{\2}$ attain large values.

For the a posteriori parameter choice based on the sequential discrepancy principle, we can also establish an error estimate with clear dependence of the coefficient $\eta$.
\begin{thm}\label{th_2:sdp}
Under the conditions of Theorem~\ref{th_2:1} consider the regularization parameter choice based on the sequential discrepancy principle.
Suppose $\alpha_*=\alpha_{j_*} \in \Delta_\nu$ is chosen according to
$$\|Ax_{\alpha_{j_*},\eta}^\delta-\yd\| \le \tau\,\delta <  \|Ax_{\alpha_{(j_*-1)},\eta}^\delta-\yd\| .$$
Then the $\ell^1$ error estimate
\begin{align*}
\|x_{\alpha_{j_*},\eta}^{\delta}-x^\dag\|_{\ell^1} & \leq  C_2 \delta
\end{align*}
holds true with the constant
$$C_2 =\rho \,c_0(\eta)\,
\left[\left(\frac{2}{\tau-1}\right)^{q-1}\frac{\tau^{q}+1}{\tau^{q}-1}\frac{1}{\nu}+\tau+1\right]
.$$
%c_0(\eta)=\frac{1}{j_{|I|}(A)}\,\left(\frac{1}{\rho}+\frac{1}{\underline \kappa} +\eta\|x^\dagger\|_{\ell^\infty}\right)
\end{thm}
\begin{proof}
The inequality (\ref{eq_2:1}) provides us with
\begin{equation}\label{ine:x1}
\begin{aligned}
&\|x-x^\dag\|_{\ell^1} \leq  \rho\left[\Omega(x)-\Omega(x^\dag)
      + c_0 (\eta)\|Ax-Ax^\dag\|\right]\\
\end{aligned}
\end{equation}
and
\begin{equation}\label{DP:1}
\begin{aligned}
\Omega(x^\dag)-\Omega(x)\leq
      c_0 (\eta)\|Ax-Ax^\dag\|.
\end{aligned}
\end{equation}
Moreover, by fixing any parameter set $(\eta,\alpha,\delta)$ we have for the minimizer $x_{\alpha,\eta}^\delta$ that
\[
\frac{1}{q}\|A\xbe-y^{\delta}\|^{q}+\alpha\Omega(\xbe)\leq\frac{\delta^q}{q}
       + \alpha\Omega(x^\dag),
\]
and consequently
\begin{equation}\label{ine:ome}
\begin{aligned}
\Omega(\xbe)-\Omega_{\alpha}(x^\dag)\leq \frac{\delta^q}{q\alpha}.
\end{aligned}
\end{equation}

Choosing another parameter $\alpha_{j_*-1}$, we derive
\begin{align}\label{DP:2}
\frac{1}{q}\|Ax_{\alpha_{j_*-1},\eta}^{\delta}-y^{\delta}\|^{q}
& \leq \frac{1}{q}\|Ax^{\dag}-y^{\delta}\|^{q}
+ \alpha_{j_*-1}\left(\Omega(x^\dag)-\Omega(x_{\alpha_{j_*-1},\eta}^{\delta})\right) \nonumber \\
& \stackrel{(\ref{DP:1})}{\leq} \frac{\delta^q}{q}
+\frac{c_0 (\eta)\alpha_{j_*}}{\nu}\|Ax_{\alpha_{j_*-1},\eta}^{\delta}-Ax^\dag\|.
\end{align}
Noticing the fact that $\|Ax_{\alpha_{j_*-1},\eta}^{\delta}-y^\delta\|>\tau \delta$, we thus obtain
\begin{equation}\label{eq_thm31deltainequal}
\begin{aligned}
\frac{\delta^{q}}{q}
\leq \frac{c_0(\eta)\alpha_{j_*}}{(\tau^{q}-1)\nu}\|A x_{\alpha_{j_*-1},\eta}^{\delta}-A x^\dag\|.
\end{aligned}
\end{equation}
On the other hand, by $(a+b)^q\leq 2^{q-1}(a^q+b^q), a,b\geq 0, q\geq 1$ and the inequalities (\ref{DP:2})-(\ref{eq_thm31deltainequal}), we can derive
\begin{equation*}
\begin{aligned}
\frac{\|Ax_{\alpha_{j_*-1},\eta}^{\delta}-Ax^{\dag}\|^q}{q} & \leq 2^{q-1}\left(\frac{\delta^q}{q} +
\frac{1}{q}\|Ax_{\alpha_{j_*-1},\eta}^{\delta}-y^{\delta}\|^{q}\right)\\
&\stackrel{(\ref{DP:2})}{\leq} 2^q\frac{\delta^q}{q}
+2^{q-1}\frac{c_0(\eta)\alpha_{j_*}}{\nu}\|Ax_{\alpha_{j_*-1},\eta}^{\delta}-Ax^\dag\|\\
&\stackrel{(\ref{eq_thm31deltainequal})}{\leq} 2^{q-1}\frac{\tau^{q}+1}{\tau^{q}-1}\frac{c_0(\eta)\alpha_{j_*}}{\nu}\|Ax_{\alpha_{j_*-1},\eta}^{\delta}-Ax^\dag\|,
\end{aligned}
\end{equation*}
that is,
\begin{equation*}
\begin{aligned}
\frac{c_0(\eta)\alpha_{j_*}}{\nu}
\geq\frac{\tau^{q}-1}{\tau^{q}+1}\frac{1}{2^{q-1}q}\|Ax_{\alpha_{j_*-1},\eta}^{\delta}-Ax^\dag\|^{q-1}.
\end{aligned}
\end{equation*}
Substituting $(\tau-1)\delta\leq\|Ax_{\alpha_{j_*-1},\eta}^{\delta}-Ax^{\dag}\|$ into previous inequality, we can derive
\begin{equation}\label{3:pa}
\begin{aligned}
\frac{c_0(\eta)\alpha_{j_*}}{\nu}\geq  \frac{1}{q} \frac{\tau^{q}-1}{\tau^{q}+1}\left(\frac{\tau-1}{2}\right)^{q-1}\delta^{q-1}.
\end{aligned}
\end{equation}

Finally, combining \eqref{ine:x1}, \eqref{ine:ome} and \eqref{3:pa}, we obtain
\begin{equation*}
\begin{aligned}
\|x_{\alpha_{j_*},\eta}^{\delta}-x^\dag\|_{\ell^1}&\leq \rho\left[\frac{\delta^q}{q\alpha_{j_*}}+c_0(\eta)(\tau+1)\delta\right]\\
&\leq \rho c_0(\eta)\left[
\left(\frac{2}{\tau-1}\right)^{q-1}\frac{\tau^{q}+1}{(\tau^{q}-1)\nu}\delta+(\tau+1)\delta\right]
\end{aligned}
\end{equation*}
which completes the proof.
\end{proof}

As one can observe, constants near error estimates of Theorems \ref{th_2:1} and \ref{th_2:sdp} highly depend on the coefficient $\eta$ and the exact solution $x^{\dag}$. Even if one cannot know the exact solution, the
subsequent section will show that a more sparse approximate solution is expected when the factor $\eta$ becomes small.

\section{Numerical examples and extended discussion}
\label{se5}
\subsection{Numerical examples}
In this section, we provide some numerical examples verifying and illustrating the theoretical predictions of the previous sections, for linear forward operators in the first subsection, and
for a nonlinear test example in the second subsection.
\subsubsection{The discretized linear case}
The following finite-dimensional ill-posed problem in the sequence space is considered such that
$$A\mathbf{x}= \mathbf{y} + \omega = \mathbf{y}^{\delta},$$
where $\omega$ is a Gaussian random variable with zero mean and a variance matrix $\delta^2 I$.

We use the Tikhonov type functional
\begin{align}\label{eq_TikBayFun}
\Phi_{\alpha}^{\delta}({\bf x}):=\frac{1}{2}\|A{\bf x}-\mathbf{y}^{\delta}\|^{2}+ \alpha \left(\|{\bf x}\|_{\ell^0}+ \mathcal{R}({\bf x})\right),
\end{align}
with $\mathcal{R}(x):=\frac{\eta}{2} \|\mathbf{x}\|_{\ell^2}^2$ or $\mathcal{R}(x):=\frac{\eta}{p} \|\mathbf{x}\|_{\ell^p}^p\;(1<p<2)$ respectively.
Noticing that the difficulty of obtaining the minimizer of $\Phi_{\alpha}^{\delta}$ remains,  we adopt the Tikhonov type functional (\ref{eq_TikBayFun}) into Bayesian interfaces as in \cite{WLMC13}. More precisely, we introduce the prior density, in the finite-dimensional regime,
\begin{align*}
p(\mathbf{x})\propto \exp\left(- \tilde{\alpha} \left(\|{\bf x}\|_{\ell^0}+ \mathcal{R}({\bf x})\right)\right).
\end{align*}
The posterior distribution then has a density
\begin{equation}\label{eq_ppdf}
\begin{aligned}
p(\mathbf{x}|\mathbf{y}^\delta)\propto\exp\left(-\frac{1}{2\delta^2}\|A\mathbf{x}-\mathbf{y}^\delta\|^2
- \tilde{\alpha} \left(\|{\bf x}\|_{\ell^0}+ \mathcal{R}({\bf x})\right)\right).
\end{aligned}
\end{equation}
Then the maximum a posteriori (MAP) estimate $\hat{\mathbf{x}}$ of (\ref{eq_ppdf}) coincides with the minimizer of (\ref{eq_TikBayFun}) with $\alpha = \tilde{\alpha}\delta^2$. To better illustrate the influence of the convex/stabilizing functional $\mathcal{R}(x)$, as mentioned above, we consider the
following regularization schemes
\begin{equation}\label{eq_twoRx}
\begin{aligned}
& \Phi_{\alpha,2}^{\delta}({\bf x}):= \frac{1}{2}\|A\mathbf{x}-\mathbf{y}^\delta\|^2+ \alpha\left(\|\mathbf{x}\|_{\ell^0}+\frac{\eta}{2}\|\mathbf{x}\|_{\ell^2}^2\right);\\
& \Phi_{\alpha,1.1}^{\delta}({\bf x}):= \frac{1}{2}\|A\mathbf{x}-\mathbf{y}^\delta\|^2+ \alpha\left(\|\mathbf{x}\|_{\ell^0}+\frac{\eta}{1.1}\|\mathbf{x}\|_{\ell^{1.1}}^{1.1}\right).
\end{aligned}
\end{equation}
We shall emphasize that the latter penalty functional $\Omega(x)=\|\mathbf{x}\|_{\ell^0}+\frac{\eta}{1.1}\|\mathbf{x}\|_{\ell^{1.1}}^{1.1}$ has a better sparsity promoting properties compared with the former one.

To obtain a single approximant from the posterior distribution density $p(\mathbf{x}|\mathbf{y}^\delta)$ given a known observation $\mathbf{y}^\delta$, we can either choose an MAP estimate or a posterior mean (PM) estimate where the latter one needs to integrate the posterior probability densities numerically. In current section, we implement the Gibbs sampler to construct transition kernels in the MCMC method. For extended discussion on statistical inverse problems and MCMC methods, we refer to the monograph \cite{bayes05} and references therein.

The first example considers the numerical differentiation solving the Volterra integral equation of the first kind
\begin{align*}
Ax(s) = \int_0^s x(t)dt, \, s\in (0,1)
\end{align*}
where the solution $x$ is equally discretized with $128$ mesh points in the interval $(0,1)$ and the exact solution vector ${\bf x}$ contains $128$ coefficients with $6$ non-zero ones.
The noise $\omega$ is added to the exact data ${\bf y}$  with a noise covariance $\delta^2 = 10^{-6}$.
The regularization parameter $\tilde{\alpha}$ in the posterior distribution density $p(\mathbf{x}|\mathbf{y}^\delta)$ is chosen in a regularization parameter set $\{2^{-k}/\delta^2,k=0,1,\ldots\}$ where the chosen $\tilde{\alpha}$ provides a slightly larger residual with respect to the noise level. Recalling Theorem \ref{th_2:sdp}, such a parameter choice rule obeys the sequential discrepancy principle there.
In order to keep the numerical realization consistent to each other, we use the Gibbs sampler (see e.g. \cite[Section 3.6.3]{bayes05}) to obtain both minimizers of the two regularization schemes in (\ref{eq_twoRx}). The sample size of the MCMC methods is chosen by $10^5$. Except the MAP estimate $\hat{x}_{MAP}$, we also generate the PM estimate $\hat{x}_{PM}$ by choosing a suitable burn-in period $[0,5000]$.
We summarize the quantity information for both MAP and PM estimates in Table \ref{tab:1} where the $\ell^2$ and $\ell^1$ relative errors take the standard form, for instance, with $\|\hat{x}-x^{\dag}\|_{\ell^2}/\|x^{\dag}\|_{\ell^2}$ and  $\|\hat{x}-x^{\dag}\|_{\ell^1}/\|x^{\dag}\|_{\ell^1}$. Here $\hat{x}$ represents either the MAP or PM estimate.
As Table \ref{tab:1} shows, the smaller $\eta$ the more sparse estimates one can obtain. On the other hand, Theorem \ref{th_2:sdp} has asserted that if one chooses the prescribed constant $\eta$ too large or too small, the constant near the $\ell^1$-norm error estimate becomes large for the penalty functional $\Omega(x)=\|\mathbf{x}\|_{\ell^0}+\frac{\eta}{2}\|\mathbf{x}\|_{\ell^{2}}^{2}$. Such observation is mostly verified by the quantity information in Table \ref{tab:1}. For instance, the PM estimate has a minimal $\ell^1$ relative error by choosing $\eta=1/2^4$. Though it is not easy to prove the same assertion for the stabilizing functional $\frac{\eta}{1.1}\|\mathbf{x}\|_{\ell^{1.1}}^{1.1}$, the numerical performance is analogous to that of the quadratic stabilizing one by varying $\eta$.

\begin{table}[htp]
 \caption{Numerical results of the first example with respect to the PM and MAP estimates.}\label{tab:1}
 The quantity information for the PM estimate.
    \begin {center}
    \begin{tabular}{cccccc}
     \hline
 &  Relative $\ell^2$ error &
   Relative  $\ell^1$ error &  $\|\hat{x}_{PM}\|_{\ell^0}$   \\
  \hline
%% $\ell^1$ scheme  & 5.25e-3 &   4.58e-p3   &   18\\
$\ell^0+\ell^2$ scheme, $\eta/2 = 1$&  1.76e-1 &    3.64e-1    &  60\\
 $\ell^0+\ell^2$ scheme, $\eta/2 =1/2^{4}$& 5.11e-2 &    4.98e-2    &  6\\
 $\ell^0+\ell^2$ scheme, $\eta/2 = 1/2^{8}$ & 6.03e-2 &    5.76e-2    &  6\\
  $\ell^0+\ell^{1.1}$ scheme, $\eta/1.1 = 5$ & 1.10e-1  &9.75e-2 &  7 \\
 $\ell^0+\ell^{1.1}$ scheme, $\eta/1.1 = 1$ & 5.45e-2  &5.08e-2 &  6\\
$\ell^0+\ell^{1.1}$ scheme, $\eta/1.1 = 1/2^{4}$ &  5.71e-2 &    5.27e-2   & 6\\
$\ell^0+\ell^{1.1}$ scheme, $\eta/1.1 = 1/2^{8}$ &  6.28e-2 &    5.59e-2   & 6\\
\hline
    \end{tabular}\\[5mm]
    \end{center}
 The quantity information for the MAP estimate.
    \begin {center}
    \begin{tabular}{cccccc}
     \hline
 &  Relative $\ell^2$ error &
   Relative  $\ell^1$ error &  $\|\hat{x}_{MAP}\|_{\ell^0}$   \\
  \hline
%% $\ell^1$ scheme  & 5.25e-3 &   4.58e-p3   &   18\\p
$\ell^0+\ell^2$ scheme, $\eta/2 = 1$&  1.76e-1 &   2.86e-1   & 15\\
 $\ell^0+\ell^2$ scheme, $\eta/2 =1/2^{4}$& 5.62e-2 & 5.68e-2 & 6\\
 $\ell^0+\ell^2$ scheme, $\eta/2 = 1/2^{8}$ & 5.97e-2 &5.64e-2 &  6\\
  $\ell^0+\ell^{1.1}$ scheme, $\eta/1.1 = 5$ & 1.13e-1& 1.10e-1 & 7 \\
 $\ell^0+\ell^{1.1}$ scheme, $\eta/1.1 = 1$ &6.72e-2&6.40e-2&6\\
$\ell^0+\ell^{1.1}$ scheme, $\eta/1.1 = 1/2^{4}$ & 5.51e-2 & 5.00e-2 & 6\\
$\ell^0+\ell^{1.1}$ scheme, $\eta/1.1 = 1/2^{8}$ & 5.85e-2 & 5.54e-2 & 6\\
\hline
    \end{tabular}\\[5mm]
    \end{center}
 \end{table}

We select two particular choices of $\eta/p=1$, $\eta/p =1/2^4$ with $p=2$, $p=1.1$ and plot a logarithmic scale of the $\ell^2$ (and $\ell^1$) absolute error with respect to the absolute noise level~$\delta$ in Figure \ref{fig:ex1}. The reference black dashed line of order $O(\delta)$ is provided as well to verify the decaying speed of the numerical error estimates. Both choices of $\eta$ provide similar decaying speeds for small noise level but slightly vary when the noise level becomes large.
\begin{figure}[!ht]
  \centering
  \includegraphics[width=0.49\textwidth]{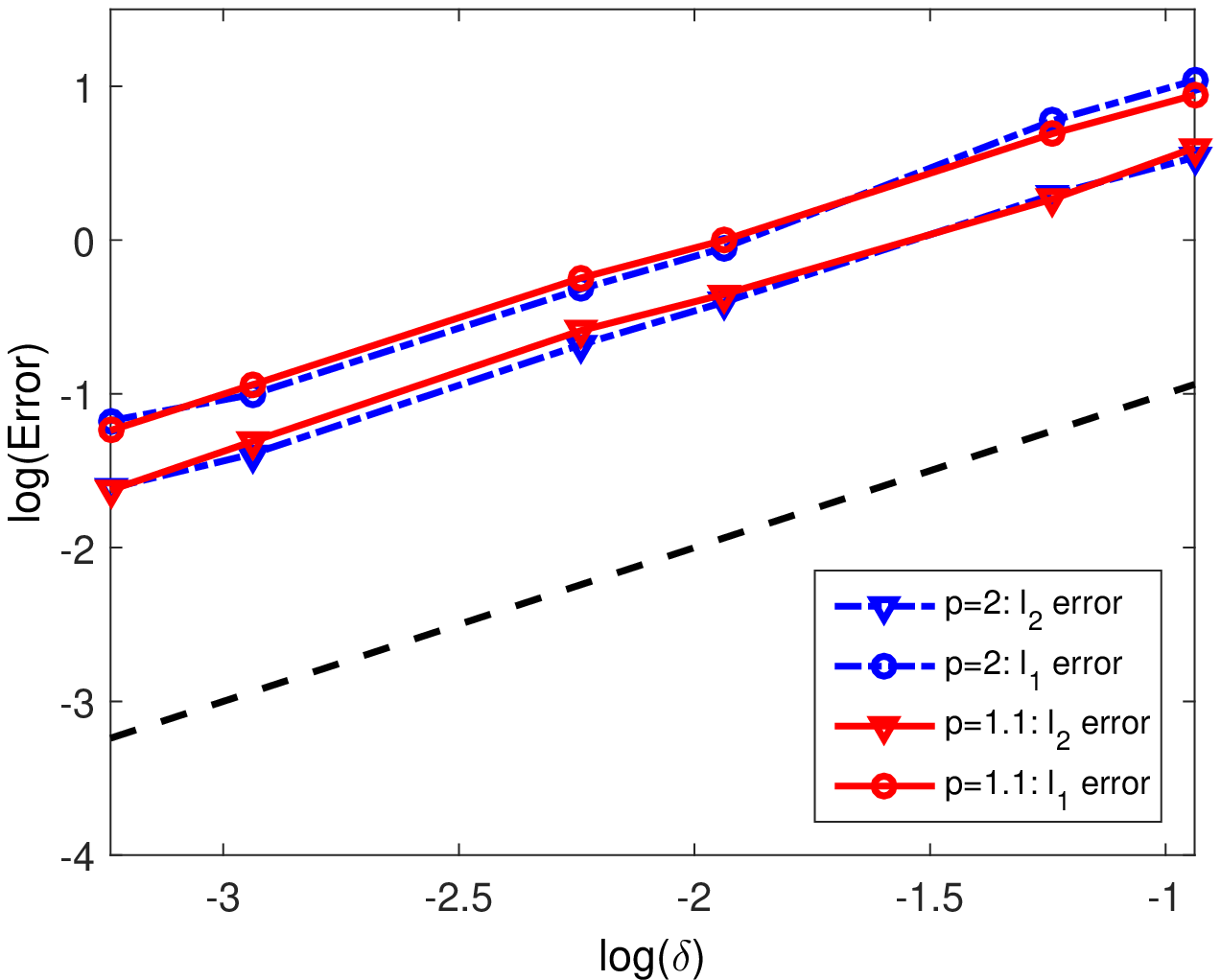}
  \includegraphics[width=0.49\textwidth]{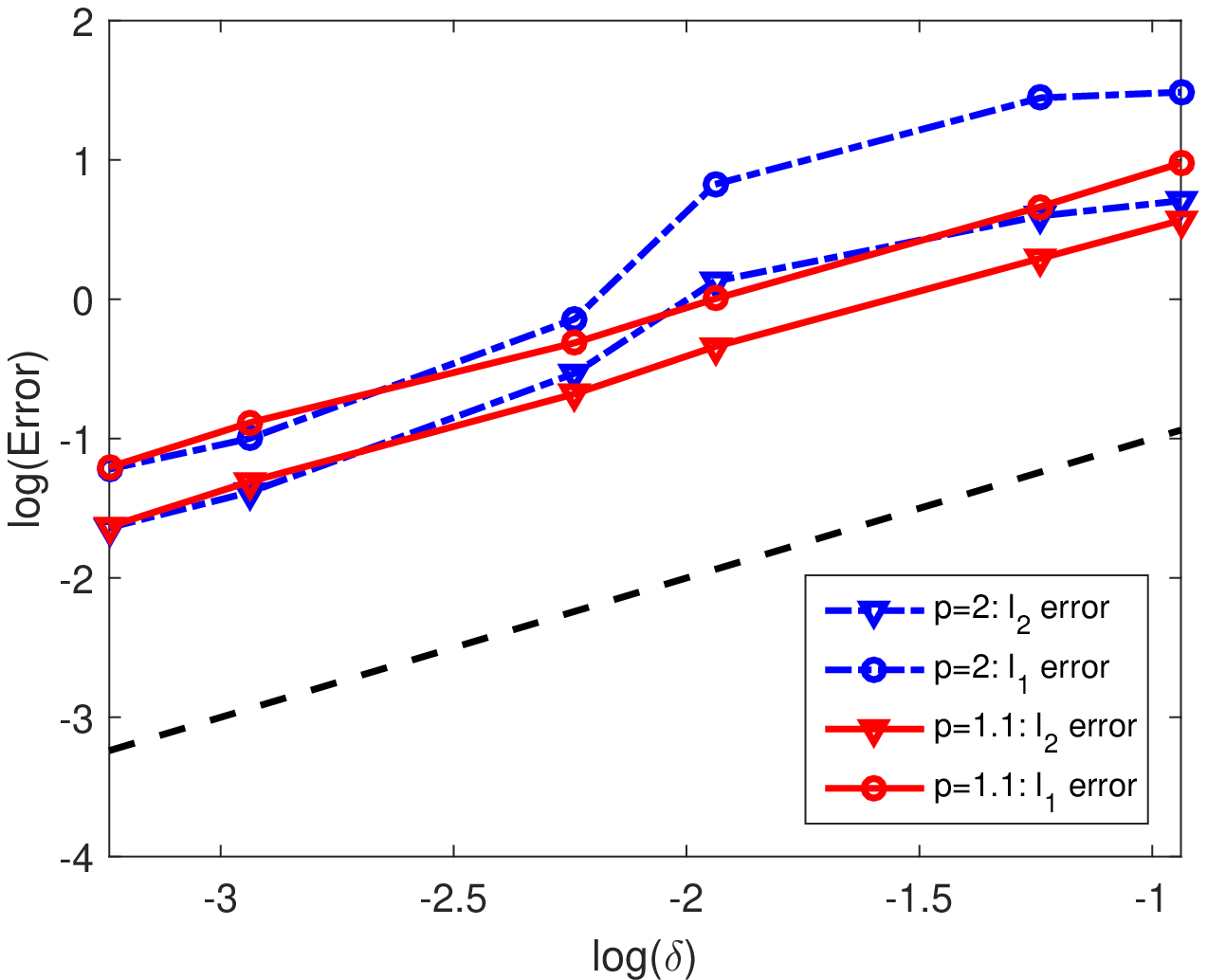}\\
  \caption{First example (PM): Absolute error versus noise level in logarithmic scales. Black dashed line is the reference line of order~$\delta$, Left: $\eta/p=1/2^4$, right: $\eta/p=1$. }\label{fig:ex1}
\end{figure}
\begin{figure}[!ht]
  \centering
  \includegraphics[width=0.49\textwidth]{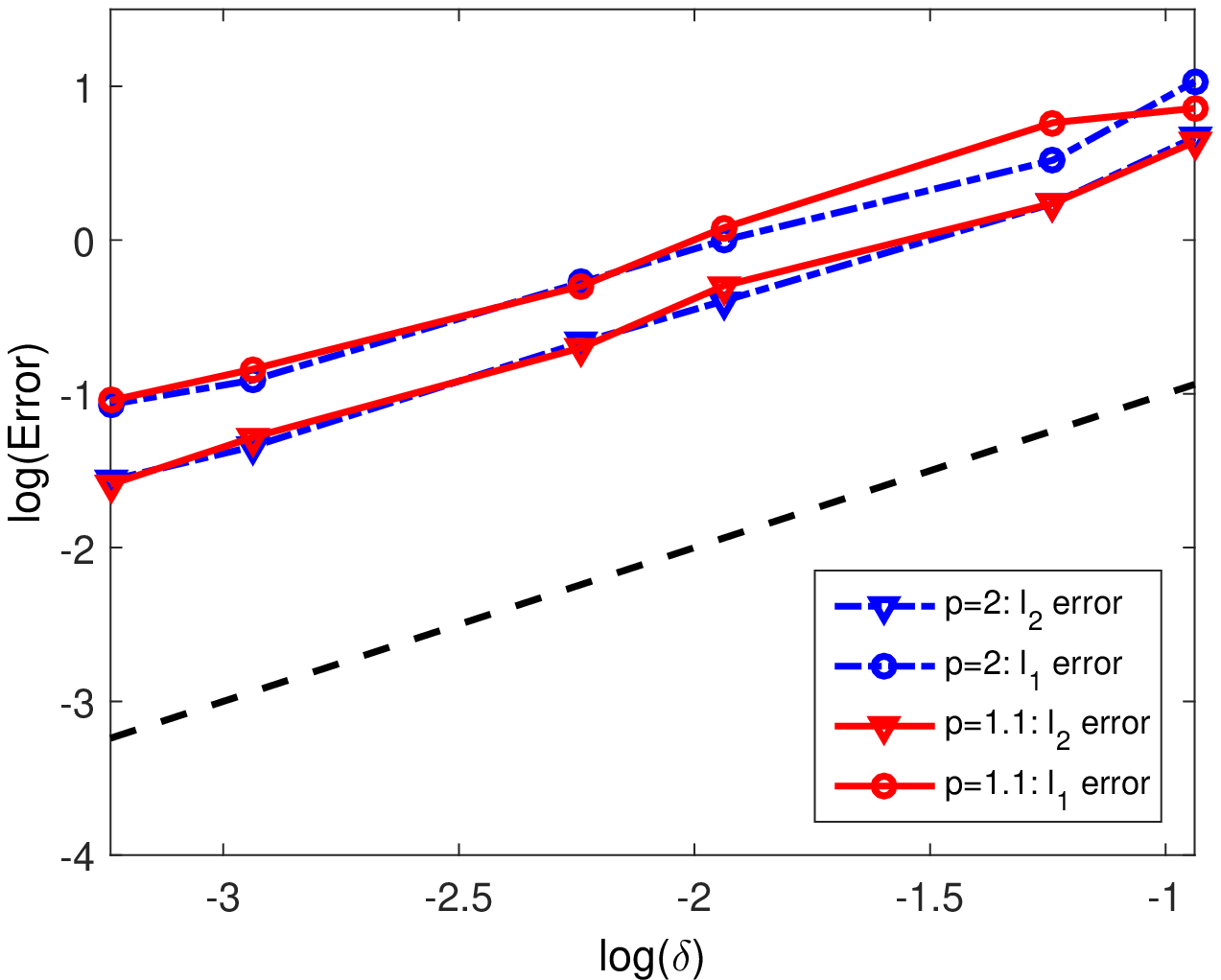}
  \includegraphics[width=0.49\textwidth]{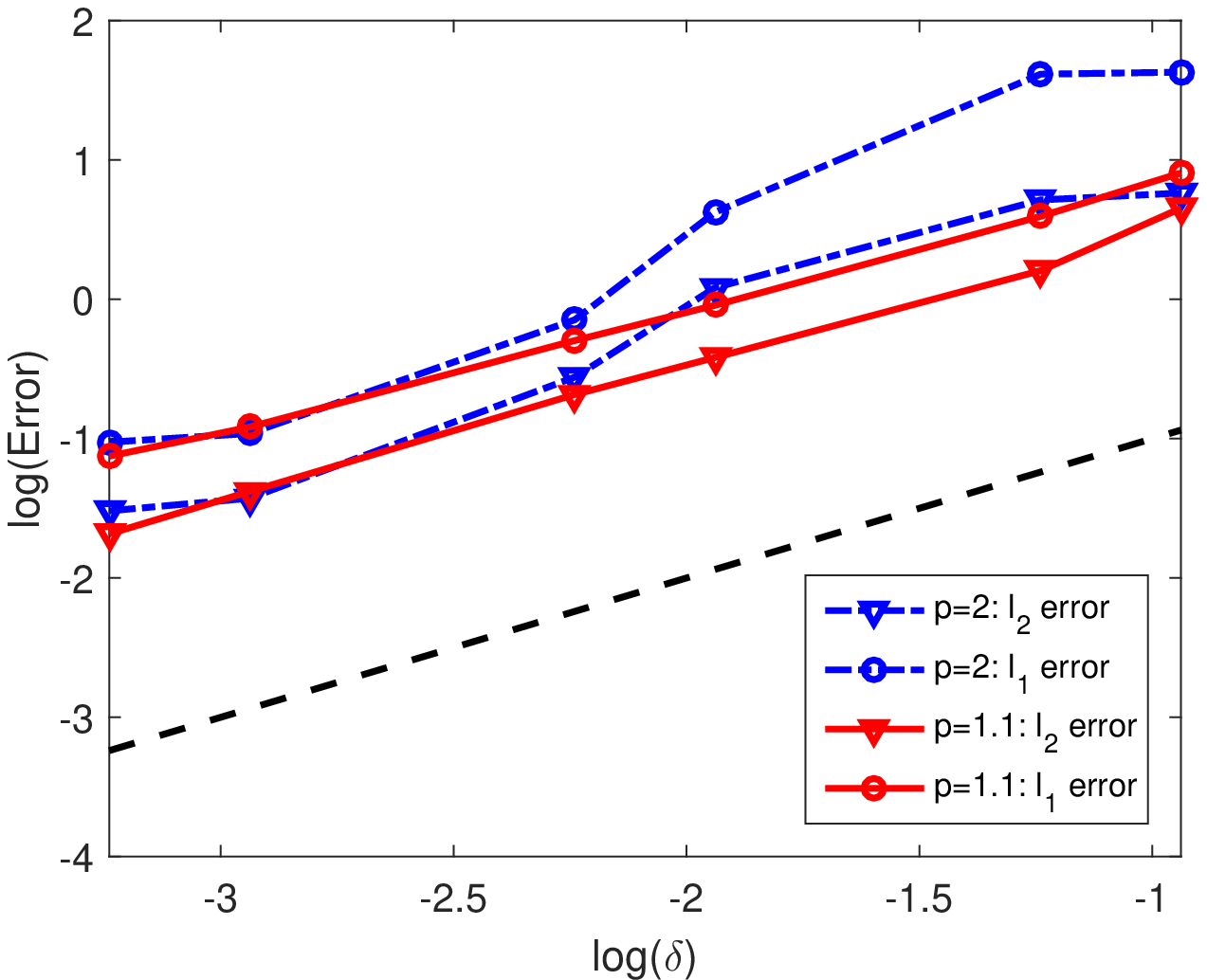}\\
  \caption{First example (MAP): Absolute error versus noise level in logarithmic scales. Black dashed line is the reference line of order~$\delta$, Left: $\eta/p=1/2^4$, right: $\eta/p=1$. }\label{fig:ex1b}
\end{figure}

\subsubsection{A nonlinear test example}

We also tested comparable approximate solutions for a discretized nonlinear operator equation $F(x) = y$, with a nonlinear Hammerstein forward operator
\begin{align*}%\label{eq_hammerstein}
F(x)(s) = \int_0^s x^2(t) dt, \quad s\in (0,1).
\end{align*}
The regularized solutions are minimizers to a discretized version of the nonlinear analog
$$\Phi_{\alpha}^\delta(x):= \frac{1}{q}\|F(x)-y^\delta\|^q + \alpha\,\Omega(x) \to \min, \quad \mbox{subject to}\quad x \in X=\ell^2,$$
to the extremal problem (\ref{eq:Tiktwo}) with penalties $\Omega(x)=\|x\|_{\ell^0}+\frac{\eta}{p}\|x\|^p_{\ell^p}$ for situations $p=2$ and $1<p<2$.

This time, we implement a wavelet expansion to transform the exact solution $x$ into a sequential space.
More precisely, given the fixed orthonormal wavelet system $\{\varphi,
\psi\}$, the exact solution $x$ can be expanded in the following form
\begin{equation*}
\begin{aligned}
x &= \sum_{k\in \mathbb{Z}}\langle x, \phi_{0,k} \rangle \phi_{0,k} + \sum_j^{\infty}\sum_{k\in \mathbb{Z}}\langle x, \psi_{j,k} \rangle \psi_{j,k} =: \sum_{\lambda \in \Lambda}\langle x, \Phi_{\lambda} \rangle \Phi_{\lambda},
\end{aligned}
\end{equation*}
with an appropriately chosen
index set $\Lambda$. The sequential solution ${\bf x}$ is presented by
$${\bf x}=\{\langle
 x,\Phi_{\lambda}\rangle\}_{\lambda\in\Lambda}. $$
Choosing the Haar wavelets, we discretize the exact solution with $64$ coefficients and only $6$ of them are non-zero. Similar to the linear example, we summarize the quantity information in Table \ref{tab:2}. The logarithmic-scale fitting curve between the $\ell^2$ (and $\ell^1$) absolute error and the absolute noise level~$\delta$ are displayed in Figure \ref{fig:ex2}. Different from the linear cases, the smaller $\eta$, the better approximants towards the exact solution. One may guess that in the nonlinear cases, the balance between the stabilizing functional $\mathcal{R}$ and the $\ell^0$-term in the penalty is more sophisticated and shall be analyzed more carefully.

%Similar conclusion can be obtained referring to the above linear example.

\begin{table}[htp]
 \caption{Numerical results of the second example with respect to the PM and MAP estimates.}\label{tab:2}
 The quantity information for the PM estimate.
    \begin {center}
    \begin{tabular}{cccccc}
     \hline
 &  Relative $\ell^2$ error &
   Relative  $\ell^1$ error &  $\|\hat{x}_{PM}\|_{\ell^0}$   \\
  \hline
$\ell^0+\ell^2$ scheme, $\eta/2 = 5$  & 5.93e-2 &   4.49e-2   &   28\\
$\ell^0+\ell^2$ scheme, $\eta/2  = 1$&  3.72e-2 &    2.97e-2    &  6\\
 $\ell^0+\ell^2$ scheme, $\eta/2  =1/2^{4}$& 2.55e-2 &    2.00e-2    &  6\\
% $\ell^0+\ell^2$ scheme, $\eta = 1/2^{8}$ & 1.17e-2 &    8.61e-3    &  6\\
  $\ell^0+\ell^{1.1}$ scheme, $\eta/1.1 = 5$ & 3.17e-2  & 2.20e-2 & 6 \\
 $\ell^0+\ell^{1.1}$ scheme, $\eta/1.1  = 1$ & 2.56e-2  &2.02e-2 & 6\\
$\ell^0+\ell^{1.1}$ scheme, $\eta/1.1 = 1/2^{4}$ &  1.44e-2 &    1.06e-2   & 6\\
\hline
    \end{tabular}\\[5mm]
    \end{center}
 The quantity information for the MAP estimate.
    \begin {center}
    \begin{tabular}{cccccc}
     \hline
 &  Relative $\ell^2$ error &
   Relative  $\ell^1$ error &  $\|\hat{x}_{MAP}\|_{\ell^0}$   \\
  \hline
$\ell^0+\ell^2$ scheme, $\eta/2  = 5$ & 5.87e-3 &   3.82e-2   &   8\\
$\ell^0+\ell^2$ scheme, $\eta/2  = 1$&  3.74e-2&   2.94e-2    & 6\\
 $\ell^0+\ell^2$ scheme, $\eta/2  =1/2^{4}$& 2.57e-2 & 2.00e-2 & 6\\
% $\ell^0+\ell^2$ scheme, $\eta = 1/2^{8}$ & *** & *** &  ***\\
  $\ell^0+\ell^{1.1}$ scheme, $\eta/1.1=5$ & 2.89e-2 & 2.18e-2& 6 \\
 $\ell^0+\ell^{1.1}$ scheme, $\eta/1.1=1$ & 2.61e-2 & 2.03e-2 & 6\\
$\ell^0+\ell^{1.1}$ scheme, $\eta/1.1=1/2^{4}$ &  1.56e-2& 1.18e-2 & 6\\
\hline
    \end{tabular}\\[5mm]
    \end{center}
 \end{table}

%The figure also indicated that for a large noise level, the results of $\ell^0+\ell^2$ is better than the results of %$\ell^0+\ell^{1.1}$.
\begin{figure}[!ht]
  \centering
  \includegraphics[width=0.49\textwidth]{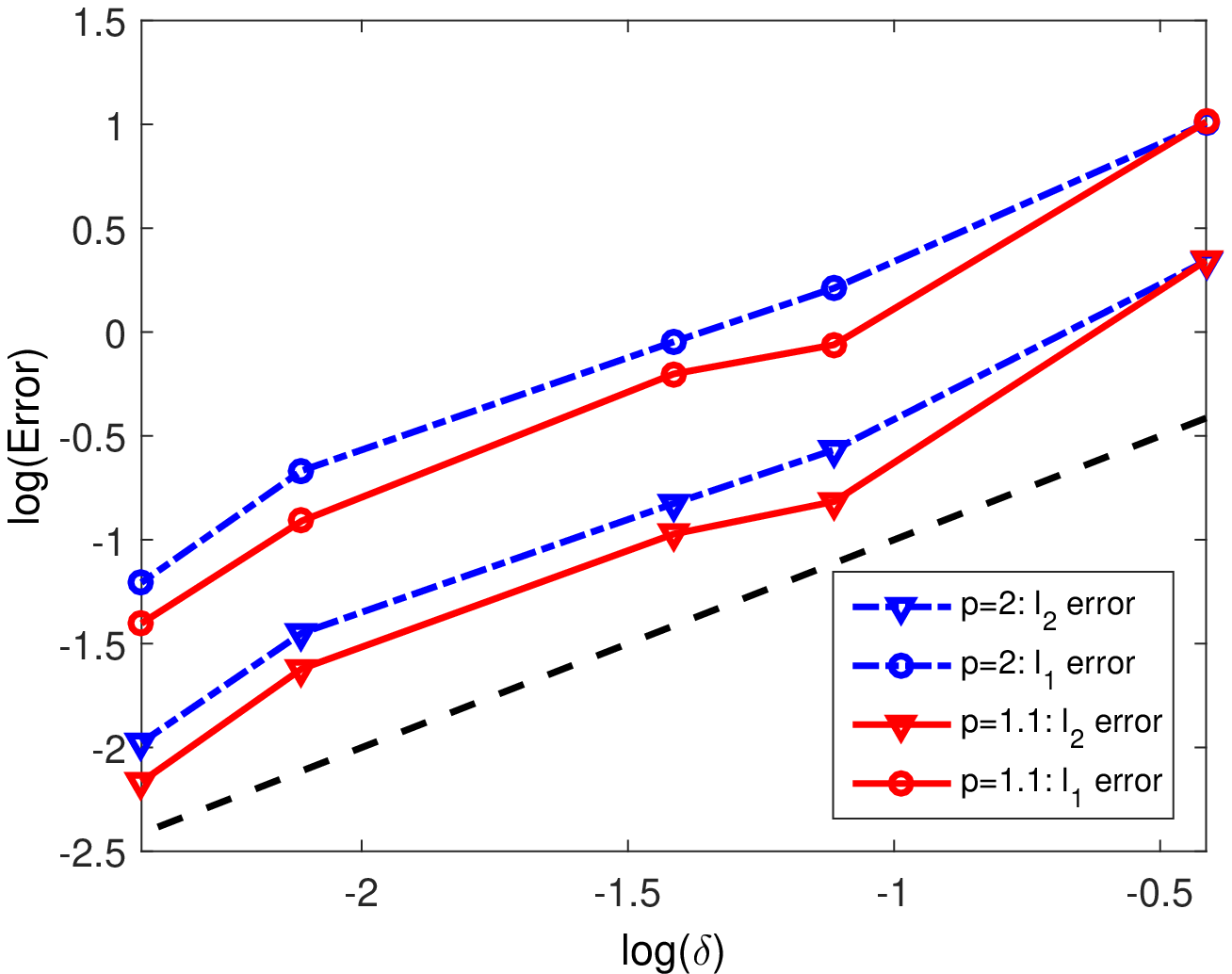}
  \includegraphics[width=0.49\textwidth]{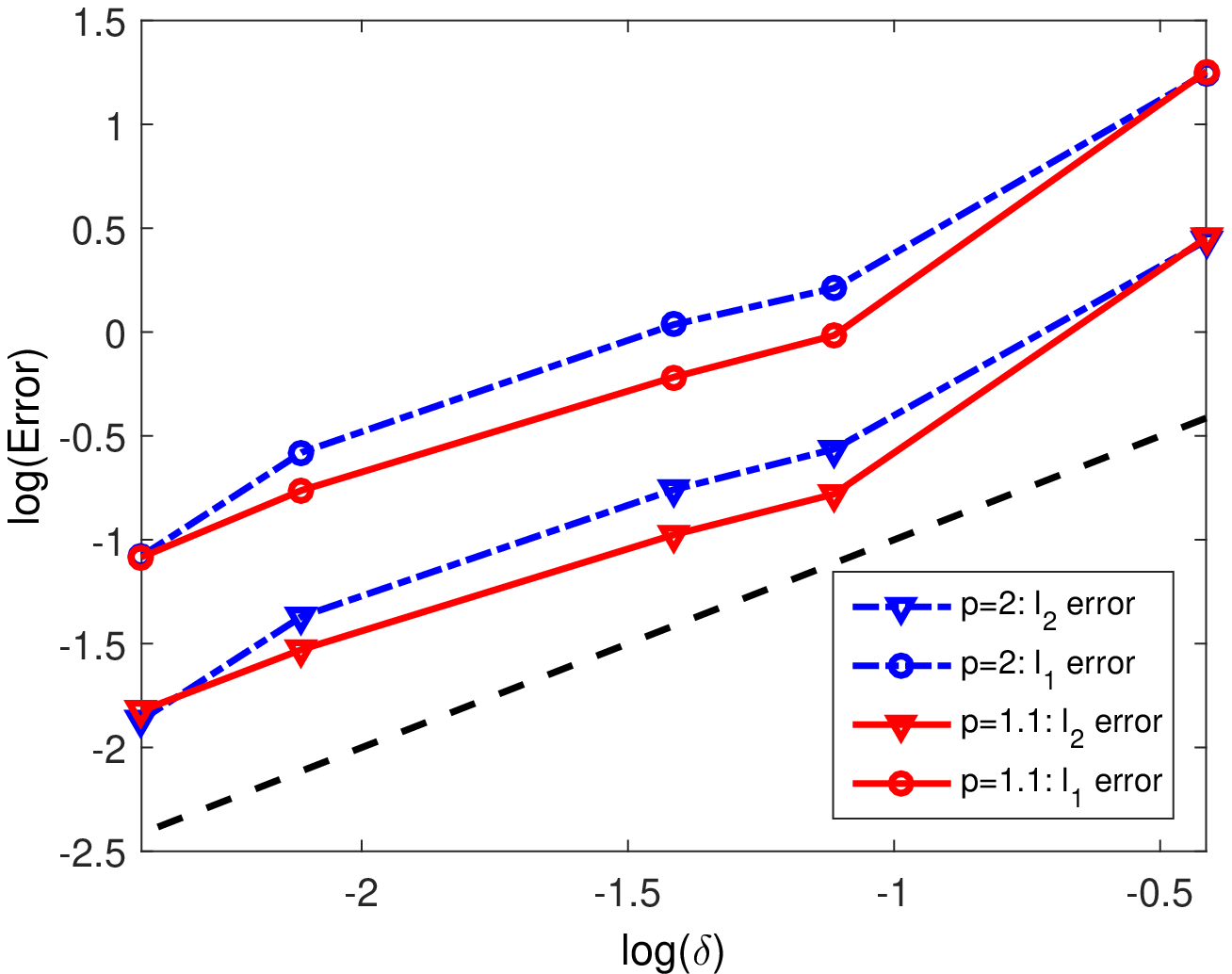}
  \caption{Second example: Absolute error versus noise level in logarithmic scales when $\eta/p=1/2^4$. Black dashed line is the reference line of order~$\delta$. Left: PM estimate; Right: MAP estimate. }\label{fig:ex2}
\end{figure}

\subsection{Extended discussion}
We establish $\ell^1$-norm error estimates for a Tikhonov type regularization with an $\ell^0$-term and a stabilizing convex functional. The error estimates are obtained by the variational inequality in Lemma \ref{lem_2:1}.
On the other hand, sparsity of $\xdag$ expresses, in some sense, a well-posed situation for solving the operator equation (\ref{eq:opeq}), although the range of $A$ is assumed to be non-closed indicating ill-posedness.
This effect becomes clear
if one considers the {\it conditional stability estimate}
\begin{equation*}
\|x-\xdag\|_{\1} \le \frac{1}{j(A,\Sigma(\xdag))}\,\|A(x-\xdag)\| \quad \mbox{for all} \quad x \in \Sigma(\xdag).
\end{equation*}
If it would be possible to ensure that the regularized solutions $\xad$ belong to the stability set  $\Sigma(\xdag)$ for sufficiently small $\delta>0$, then the combination of conditional stability and regularization
suggested originally in the article \cite{ChengYam00} (see also the more recent papers \cite{CHL14,HofYam10}) could be used for obtaining linear convergence rates when the regularization parameter
$\alpha_*>0$ is chosen a priori as $\alpha_*=\alpha_*(\delta)$ satisfying the inequalities
\begin{equation}\label{eq:apriori}
\underline c \,\delta^q \le\alpha_*(\delta) \le \overline c\,\delta^q
\end{equation}
for constants $0<\underline c \le \overline c < \infty$.
Evidently in contrast to Theorems~\ref{thm:main}, \ref{th_2:1} above, where $\alpha_*(\delta)\sim \delta^{p-1}$ is chosen, in the case (\ref{eq:apriori}) the quotient $\frac{\delta^q}{\alpha_*(\delta)}$ does not tends to zero as $\delta \to 0$, but remains in the interval $[\underline c,\overline c]$.
Unfortunately, the capability of our variety of $\ell^0$-regularization is only to ensure that $\|x_{\alpha_*}^\delta\|_{\ell^0} \le \overline K<\infty$ for all regularized solutions whenever $\delta>0$ is small enough. This, however, is far away from the requirement $x_{\alpha_*}^\delta \in \Sigma(\xdag)$ of the conditional stability approach, since the distribution of the $|I|$ nonzero components in $\xdag$ is a priori completely unknown.
We take some numerical evidence of the parameter choice rule (\ref{eq:apriori}) by choosing $\alpha_* = c \delta^{p}$ with $c=0.1$ and $c=0.2$ respectively. The numerical results are displayed in Figure \ref{fig:CY}. As one can observe, though the error in the logarithmic scale does not form a straight line as in Figures \ref{fig:ex1}-\ref{fig:ex1b}, the error is still acceptable.
 Generally, both constants $c = 0.1$ and $c = 0.2$
provide comparably the same accuracy.
% Generally, the smaller constant $c=0.1$ provides better results compared with a larger constant $c=0.2$.

\begin{figure}[!ht]
  \centering
  \includegraphics[width=0.49\textwidth]{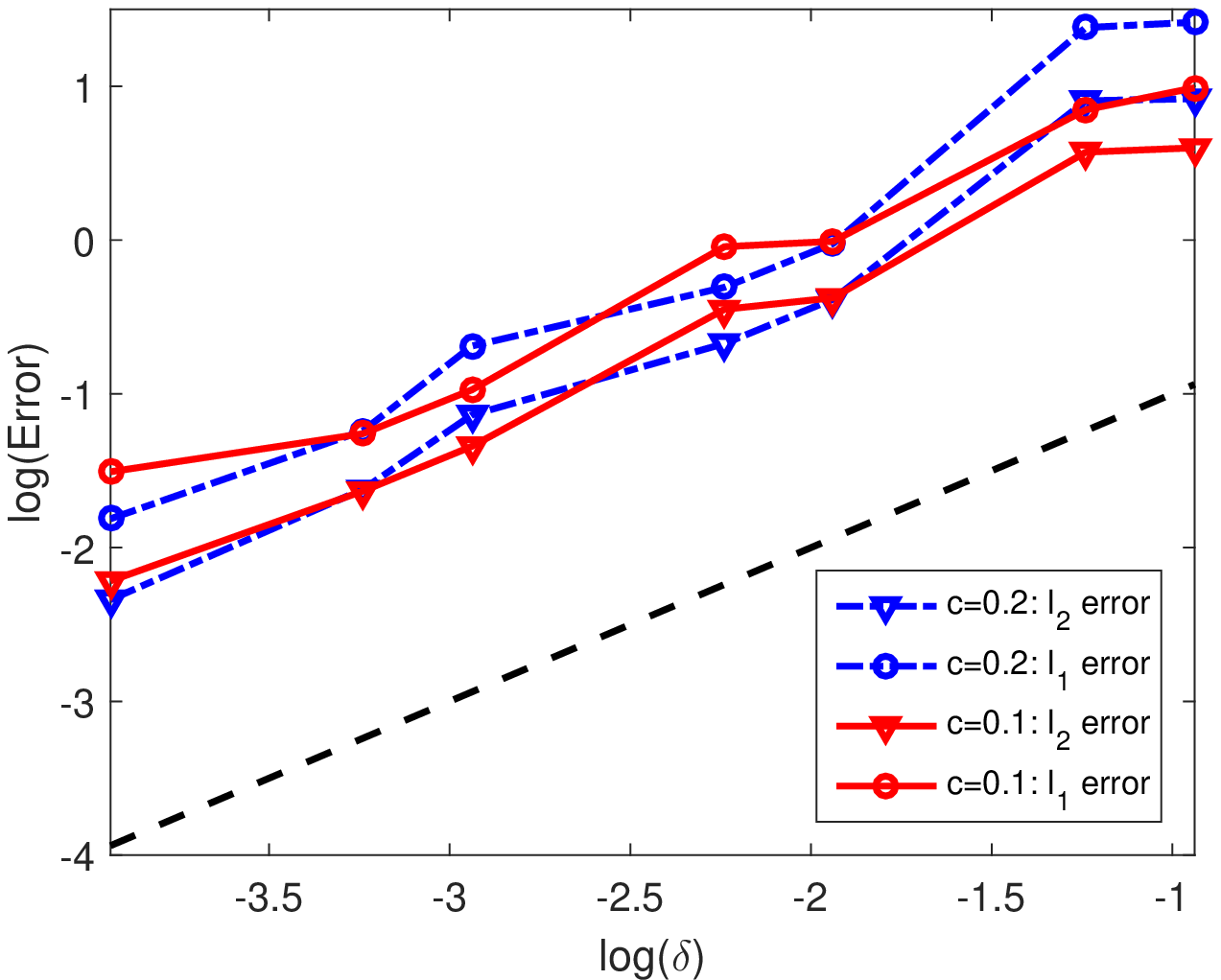}
  \includegraphics[width=0.49\textwidth]{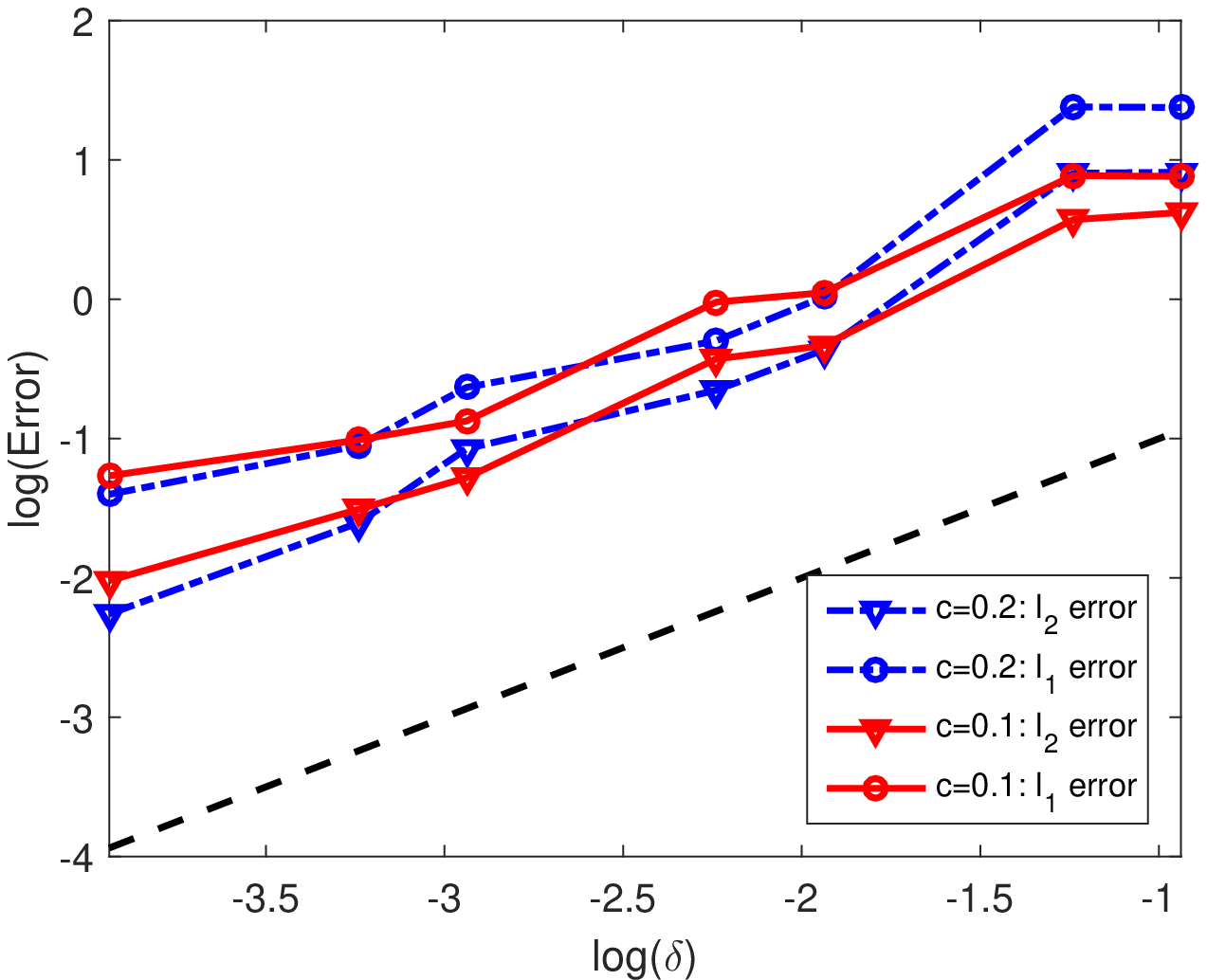}
  \caption{First example with the parameter choice rule $\alpha_* = c \delta^{p}$ and $c=0.1, 0.2$. Absolute error versus noise level in logarithmic scales when $p=2$ and $\eta/p=1/2^4$. Black dashed line is the reference line of order~$\delta$. Left: PM estimate, Right: MAP estimate.}\label{fig:CY}
\end{figure}

\section*{Acknowledgment}
W. Wang is supported by NSFC (No.11401257). S. Lu is supported by NSFC (No.11522108, 91730304), Shanghai Municipal Education Commission (No.16SG01) and Special Funds for Major State Basic Research Projects of China (2015CB856003). B.~Hofmann is supported by German Research Foundation under grant HO~1454/12-1. J. Cheng is supported by NSFC (key projects no.11331004, no.11421110002) and the Programme of Introducing Talents of Discipline to Universities (number B08018).

\end{document}